\documentclass[11pt]{amsart}

\usepackage{amsfonts}
\usepackage{amsmath,amssymb,amsthm}
\usepackage{enumerate,graphicx}
\usepackage[square,numbers]{natbib}
\usepackage{MnSymbol}
\usepackage{pst-tree}
\usepackage[crop=off]{auto-pst-pdf}
\usepackage{mathrsfs,bbm} 
\usepackage{tikz-qtree}

\newtheorem{theorem}{Theorem}[section]

\newtheorem{proposition}[theorem]{Proposition}
\newtheorem{corollary}[theorem]{Corollary}
\newtheorem{lemma}[theorem]{Lemma}

\theoremstyle{definition}

\newtheorem{definition}[theorem]{Definition}
\newtheorem{example}[theorem]{Example}

\newtheorem{problem}{Problem}
\newtheorem{remark}[theorem]{Remark}

\begin{document}

\title{Tree-Shifts: Irreducibility, mixing, and the chaos of tree-shifts}


\author{Jung-Chao Ban}
\address[Jung-Chao Ban]{Department of Applied Mathematics, National Dong Hwa University, Hualien 970003, Taiwan, ROC.}

\author{Chih-Hung Chang}
\address[Chih-Hung Chang]{Department of Applied Mathematics, National University of Kaohsiung, Kaohsiung 81148, Taiwan, ROC.}
\email{chchang@nuk.edu.tw}

\date{July 16, 2015}

\thanks{This work is partially supported by the National Science Council, ROC (Contract No NSC 102-2628-M-259-001-MY3 and 104-2115-M-390-004-).}

\baselineskip=1.2\baselineskip

\begin{abstract}
Topological behavior, such as chaos, irreducibility, and mixing of a one-sided shift of finite type, is well elucidated. Meanwhile, the investigation of multidimensional shifts, for instance, textile systems is difficult and only a few results have been obtained so far.

This paper studies shifts defined on infinite trees, which are called tree-shifts. Infinite trees have a natural structure of one-sided symbolic dynamical systems equipped with multiple shift maps and constitute an intermediate class in between one-sided shifts and multidimensional shifts. We have shown not only an irreducible tree-shift of finite type, but also a mixing tree-shift that are chaotic in the sense of Devaney. Furthermore, the graph and labeled graph representations of tree-shifts are revealed so that the verification of irreducibility and mixing of a tree-shift is equivalent to determining the irreducibility and mixing of matrices, respectively. This extends the classical results of one-sided symbolic dynamics.

A necessary and sufficient condition for the irreducibility and mixing of tree-shifts of finite type is demonstrated. Most important of all, the examination can be done in finite steps with an upper bound.
\end{abstract}

\maketitle

\section{Introduction}

Over the past few decades, there have been many researches about chaotic systems. For instance, the strange attractor in the Lorenz system, period doubling in quadratic maps, and Julia sets in complex-valued functions have been studied by scientists in many disciplines. The main reasoning is that the chaotic and random behavior of solutions of deterministic systems is an inherent feature of many nonlinear systems.

Nevertheless, for most systems, the theoretical analysis of the chaotic behavior is difficult. One of the most frequently used techniques is transferring the original system to a conjugate or semiconjugate symbolic dynamical system and then investigating the chaotic behavior in symbolic dynamics (see \cite{Dev-1987} and the references therein).

In the classical symbolic dynamical systems, shifts of finite type are an important class for investigation. A shift of finite type is a set of right-infinite or bi-infinite paths in a finite graph. Moreover, investigating the graph representation of a shift of finite type reveals some important properties such as irreducibility, mixing, and spatial chaos (see \cite{Kit-1998, LM-1995}).

In \cite{AB-TCS2012, AB-TCS2013}, the authors introduce the notion of shifts defined on infinite trees, that are called tree-shifts. Infinite trees have a natural structure of one-sided symbolic dynamical systems equipped with multiple shift maps. The $i$th shift map applies to a tree that gives the subtree rooted at the $i$th children of the tree. Sets of finite patterns of tree-shifts of finite type are strictly testable tree languages. Such testable tree languages are also called $k$-testable tree languages. Probabilistic $k$-testable models are used for pattern classification and stochastic learning. Readers are referred to \cite{VCC-PAMIIEEET2005} for more details. It is also remarkable that M\"{u}ller and Spandl show that there exists an embedding map from a topological dynamical system on metric Cantor space to a cellular automaton defined on Cayley graph, which preserves topological entropy \cite{MS-ETDS2009}.

Tree-shifts are interesting for elucidation since they constitute an intermediate class in between one-sided shifts and multidimensional shifts. The conjugacy of multidimensional shifts of finite type (also known as textile systems or tiling systems) is undecidable (see \cite{CJJ+-ETDS2003, JM-PAMS1999, LS-2002} and references therein). Namely, there is no algorithm for determining whether two tiling systems are conjugate Nevertheless, Williams indicates that the conjugacy of one-sided shifts of finite type is decidable (see \cite{LM-1995}). Aubrun and B\'{e}al extend Williams' result to tree-shifts; more precisely, they show that the conjugacy of irreducible tree-shifts of finite type is decidable \cite{AB-TCS2012}. Furthermore, Aubrun and B\'{e}al accomplish other celebrated results in tree-shifts, such as realizing tree-shifts of finite type and sofic tree-shifts via tree automata, developing an algorithm for determining whether a sofic tree-shift is a tree-shift of finite type, and the existence of irreducible sofic tree-shifts that are not the factors of tree-shifts of finite type. Readers are referred to \cite{AB-TCS2012, AB-TCS2013} for more details.

This paper is the first part of serial works about tree-shifts. In this investigation, we focus on the topological complexity of tree-shifts such as irreducibility, mixing, dense periodic points, and topological transitivity. By elaborating on the topological behavior which a tree-shift is capable of, we intend to reveal the connection between tree-shifts and nonlinear dynamical systems.

Investigating the irreducible components of dynamical systems is essential. It is known that every shift space can be decomposed into several irreducible subshift spaces, and the dynamical behavior of the whole system is determined by the dynamical behavior of these irreducible subsystems. For instance, suppose $X$ is a shift of finite type which consists of several irreducible subshifts of finite type, say, $X_1, X_2, \ldots, X_n$ for some $n$. Then the topological entropy of $X$ is the maximum of entropies of these subshifts. More explicitly, $h(X) = \max_i \{h(X_i)\}$. Readers are referred to \cite{Kit-1998, LM-1995} for more details.

It is known that the full shift is chaotic in the sense of Devaney \cite{Dev-1987}. In other words, the full shift is topologically transitive, sensitive, and contains a set of periodic points which is dense in itself. In this paper, we investigate some topological behavior, such as chaos, irreducibility, and mixing, of tree-shifts. By demonstrating an equivalent condition for the irreducibility of a tree-shift, which is defined in \cite{AB-TCS2012, AB-TCS2013}, we extend Aubrun and B\'{e}al's definition to define the notion of mixing. Moreover, the graph representations and labeled graph representations of tree-shifts of finite type are revealed. It follows that elaborating the properties of tree-shifts of finite type is equivalent to studying their corresponding adjacency matrices and symbolic adjacency matrices, which extends the results of one-sided symbolic dynamics. Most important of all, the verification of irreducibility and mixing can be done in finite steps with an upper bound.

We show that a mixing tree-shift is chaotic. Furthermore, an irreducible tree-shift of finite type is also a chaotic system. Notably, mixing is a sufficient condition for tree-shifts being chaotic, and no further condition is needed. Whenever presumption mixing is replaced by irreducible, the tree-shift has to be of finite type to ensure the chaos. This extends the classical results in symbolic dynamics.

Remarkably, determining whether a multidimensional system is chaotic is difficult. Boyle \emph{et al.} \cite{BPS-TAMS2010} show that, for a two dimensional tiling system, block gluing is a sufficient condition for exhibiting dense periodic points, which is one of the conditions for Devaney chaos. Herein, a higher dimensional shift space is called \emph{block gluing} if there is a constant $\kappa$ such that any two patterns with distance larger than $\kappa$ can be embedded into a bigger pattern. Meanwhile, the sufficient condition for the denseness of periodic points in a multidimensional system remains unknown. Furthermore, Ban \emph{et al.} \cite{BHL+-2015} introduce two sufficient conditions for the primitivity of two-dimensional shifts of finite type. In this work, we demonstrate that mixing is a sufficient condition for a tree-shift (which may or may not be a tree-shift of finite type) to be chaotic in the sense of Devaney.

While the theory of one-dimensional symbolic dynamics is well-established, the results of multidimensional shift spaces are relatively fewer. This work elaborates tree-shifts, which constitute an intermediate category between one-sided shifts and multidimensional shifts, and characterizes the topological properties of tree-shifts systematically. Aside from accomplishing the chaos, irreducibility, and mixing of a tree-shift, we extend the theory of symbolic dynamics as follows:

\begin{enumerate}[\bf 1)]
\item Suppose $X$ is a tree-shift and $X^{[m]}$ is the $m$th higher block tree-shift of $X$ (see Definition \ref{def:higher-block-tree-shift}). Then $X$ and $X^{[m]}$ are topological conjugate for all $m \in \mathbb{N}$.
\item Every tree-shift of finite type is conjugate to a vertex tree-shift (see Definition \ref{def:matrix-shift}), and every vertex tree-shift is a Markov tree-shift.
\item Every tree-shift of finite type $X$ has a labeled graph representation together with a symbolic adjacency matrix $S$. $X$ is irreducible (resp.~mixing) if and only if $S$ is irreducible (resp.~primitive). (See Definitions \ref{def:irreducible-symbolic-matrix} and \ref{def:primitive-symbolic-matrix}.)
\item An $n \times n$ symbolic matrix $S$ is irreducible if and only if, for each $i, j$, there exists a positive integer $k_{i,j} \leq n 2^{n-1}$ such that $S^{k_{i, j}}(i, j)$ contains the formal sum of a complete prefix set. $S$ is primitive if and only if there exists a positive integer $k \leq n^3 2^{2(n-1)}$ such that, for each $i, j$, $S^k(i, j)$ contains the formal sum of a complete prefix set.
\end{enumerate}

The rest of this paper is organized as follows. Section 2 defines some notions, such as irreducibility, mixing, and periodicity, of tree-shifts. Each tree-shift of finite type being conjugate to a Markov tree-shift is also revealed therein. Section 3 investigates the chaos of tree-shifts. The graph and labeled graph representations of tree-shifts of finite are elucidated in Section 4. Beyond that, the necessary and sufficient conditions for irreducibility and mixing are also demonstrated. Section 5 ends this paper with some concluding remarks.

\section{Definitions}

This section recalls some basic definitions of symbolic dynamics on infinite trees. The nodes of infinite trees considered in this paper have a fixed number of children and are labeled in a finite alphabet. To clarify the discussion, we focus on binary trees, but all results extend to the case
of trees with $d$ children for a fixed positive integer $d$. Hence the class of classical one-sided shift spaces is a special case in the present study.

Let $\Sigma = \{0, 1\}$ and let $\Sigma^*$ be the set of words over $\Sigma$. More specifically, $\Sigma^* = \bigcup_{n \geq 0} \Sigma^n$, where $\Sigma^n = \{w_1 w_2 \cdots w_n: w_i \in \Sigma \text{ for } 1 \leq i \leq n\}$ is the set of words of length $n$ for $n \in \mathbb{N}$ and $\Sigma^0 = \{\epsilon\}$ consists of the empty word $\epsilon$. An \emph{infinite tree} $t$ over a finite alphabet $\mathcal{A}$ is a function from $\Sigma^*$ to $\mathcal{A}$. A node of an infinite tree is a word of $\Sigma^*$. The empty word relates to the root of the tree. Suppose $x$ is a node of a tree. $x$ has children $xi$ with $i \in \Sigma$. A sequence of words $(x_k)_{1 \leq k \leq n}$ is called a \emph{path} if, for all $k \leq n-1$, $x_{k+1} = x_k i_k$ for some $i_k \in \Sigma$. For the rest of this investigation, a tree is referred as an infinite tree unless otherwise stated.

Let $t$ be a tree and let $x$ be a node, we refer $t_x$ to $t(x)$ for simplicity. A subset of words $L \subset \Sigma^*$ is called \emph{prefix-closed} if each prefix of $L$ belongs to $L$. A function $u$ defined on a finite prefix-closed subset $L$ with codomain $\mathcal{A}$ is called a \emph{pattern} (or \emph{block}), and $L$ is called the \emph{support} of the pattern. A subtree of a tree $t$ rooted at a node $x$ is the tree $t'$ satisfying $t'_y = t_{xy}$ for all $y \in \Sigma^*$ such that $xy$ is a node of $t$, where $xy = x_1 \cdots x_m y_1 \cdots y_n$ means the concatenation of $x = x_1 \cdots x_m$ and $y_1 \cdots y_n$.

Suppose $n$ is a nonnegative integer. Let $\Sigma_n = \bigcup_{k = 0}^n \Sigma^k$ denote the set of words of length at most $n$. We say that a pattern $u$ is \emph{a block of height $n$} (or \emph{an $n$-block}) if the support of $u$ is $\Sigma_{n-1}$, denoted by $\mathrm{height}(u) = n$. Furthermore, $u$ is a pattern of a tree $t$ if there exists $x \in \Sigma^*$ such that $u_y = t_{xy}$ for every node $y$ of $u$. In this case, we say that $u$ is a pattern of $t$ rooted at the node $x$. A tree $t$ is said to \emph{avoid} $u$ if $u$ is not a pattern of $t$. If $u$ is a pattern of $t$, then $u$ is called an \emph{allowed pattern} of $t$.

We denote by $\mathcal{T}$ (or $\mathcal{A}^{\Sigma^*}$) the set of all infinite trees on $\mathcal{A}$. For $i \in \Sigma$, the shift transformations $\sigma_i$ from $\mathcal{T}$ to itself are defined as follows. For every tree $t \in \mathcal{T}$, $\sigma_i(t)$ is the tree rooted at the $i$th child of $t$, that is, $\sigma_i(t)_x = t_{ix}$ for all $x \in \Sigma^*$. For the purpose of simplification of the notation, we omit the parentheses and denote $\sigma_i(t)$ by $\sigma_i t$. The set $\mathcal{T}$ equipped with the shift transformations $\sigma_i$ is called the \emph{full tree-shift} of infinite trees over $\mathcal{A}$. Suppose $w = w_1 \cdots w_n \in \Sigma^*$. Define $\sigma_w = \sigma_{w_n} \circ \sigma_{w_{n-1}} \circ \cdots \circ \sigma_{w_1}$. It follows immediately that $(\sigma_w t)_x = t_{wx}$ for all $x \in \Sigma^*$.

Given a collection of patterns $\mathcal{F}$, let $\mathsf{X}_{\mathcal{F}}$ denote the set of all trees avoiding any element of $\mathcal{F}$. A subset $X \subseteq \mathcal{T}$ is called a \emph{tree-shift} if $X = \mathsf{X}_{\mathcal{F}}$ for some $\mathcal{F}$. We say that $\mathcal{F}$ is \emph{a set of forbidden patterns} (or \emph{a forbidden set}) of $X$. It can be seen that a tree-shift satisfies $\sigma_w X \subseteq X$ for all $w \in \Sigma^*$.

Denote the set of all blocks of height $n$ of $X$ by $B_n(X)$, and denote the set of all blocks of $X$ by $B(X)$. Suppose $u \in B_n(X)$ for some $n \geq 2$. Let $\sigma_i u$ be the block of height $n-1$ such that $(\sigma_i u)_x = u_{ix}$ for $x \in \Sigma_{n-2}$. The block $u$ is written as $u = (u_{\epsilon}, \sigma_0 u, \sigma_1 u)$.

A set of patterns $L$ is called \emph{factorial} if $u \in L$ and $v$ is a sub-pattern of $u$ implies $v \in L$. We say that $v$ is a sub-pattern of $u$ if $v$ is a subtree of $u$ rooted at some node $x$ of $u$. The set $L$ is called \emph{extensible} if for any pattern $u \in L$ with support $S(u)$, there exists a pattern $v \in L$ with support $S(v)$ such that $S(u) \subset S(v)$, $v$ coincides with $u$ on $S(u)$, and for any $x \in S(u)$, we have $xi \in S(v)$ for all $i \in \Sigma$.

Suppose $L$ is a factorial and extensible set of patterns. Let $\mathcal{X}(L)$ be the collection of trees whose patterns belong to $L$. Then $\mathcal{X}(L)$ is a tree-shift and $B(\mathcal{X}(L)) = L$. Conversely, if $X$ is a tree-shift, then $X = \mathcal{X}(B(X))$. This result is similar to the one known for the classical shift spaces. Readers are referred to \cite{AB-TCS2013, LM-1995} for more details.

\begin{example}\label{eg:even-sum-each-block}
Figure \ref{fig:even-sum-each-block} illustrates an infinite tree of a tree-shift $\mathsf{X}_{\mathcal{F}}$ on the alphabet $\mathcal{A} = \{0, 1\}$ defined by a finite set $\mathcal{F}$ of forbidden blocks of height 2. The forbidden blocks are those whose label sum is odd; more precisely,
$$
\mathcal{F} = \{(u_{\epsilon}, u_0, u_1): u_{\epsilon} + u_0 + u_1 = 1 \pmod{2}\}.
$$

\begin{figure}
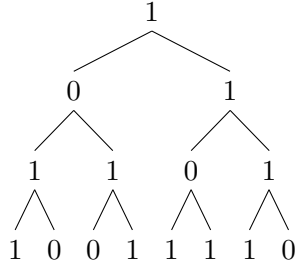

\begin{center}
\Tree [.1 [.0 [.1 1 0 ] [.1 0 1 ] ]
         [.1 [.0 1 1 ] [.1 1 0 ] ] ]
\end{center}
\caption{A part of an infinite tree of the tree-shift $\mathsf{X}_{\mathcal{F}}$, where $\mathcal{F}$ is the set of blocks of height 2 whose label sum is odd.}\label{fig:even-sum-each-block}
\end{figure}
\end{example}

Suppose $x$ is a node, that is, $x \in \Sigma^*$. Denote by $|x|$ the length of $x$. For any two trees $t$ and $t'$, define
\begin{equation}
\mathrm{d}(t, t') = \left\{
             \begin{array}{ll}
               2^{-n}, & n = \min\{|x|: t_x \neq t'_x\} < \infty \hbox{;} \\
               0, & \hbox{otherwise.}
             \end{array}
           \right.
\end{equation}
Then $\mathrm{d}$ is a metric on $\mathcal{T}$ and is similar as the metric defined on symbolic dynamics (cf.~\cite{LM-1995}).

Let $\mathcal{T}$ and $\mathcal{T}'$ be the full tree-shifts over finite alphabets $\mathcal{A}$ and $\mathcal{A}'$, respectively, and let $X$ be a tree-subshift of $\mathcal{T}$. (That is, $X$ is itself a tree-shift and $X \subseteq \mathcal{T}$.) A function $\phi: X \to \mathcal{T}'$ is called a \emph{sliding block code} if there exists a positive integer $m$ and a map $\Phi: B_m(X) \to \mathcal{A}'$ such that $\phi(t)_x = \Phi(u)$, the image of $m$-block of $t$ rooted at $x$ with respect to $\Phi$, for all $x \in \Sigma^*$. The local map $\Phi$ herein is called an \emph{$m$-block map}, and a \emph{block map} is a map which is an $m$-block map for some positive integer $m$.

In the theory of symbolic dynamics, the Curtis-Lyndon-Hedlund theorem (see \cite{Hed-MST1969}) indicates that, for two shift spaces $X$ and $Y$, a map $\phi: X \to Y$ is a sliding block code if and only if $\phi$ is continuous and $\phi \circ \sigma_x = \sigma_Y \circ \phi$. A similar discussion extends to tree-shifts; in other words, $\phi$ is a sliding block code (between tree-shifts) if and only if $\phi$ is continuous and commutes with all tree-shift maps $\sigma_i$ for $i \in \Sigma$.

If a sliding block code $\phi: X \to Y$, herein $X$ and $Y$ are tree-shifts, is onto, then $\phi$ is called a \emph{factor code} from $X$ to $Y$. A tree-shift $Y$ is a \emph{factor} of $X$ if there is a factor code from $X$ onto $Y$. If $\phi$ is one-to-one, then $\phi$ is called an \emph{embedding} of $X$ into $Y$.

A sliding block code $\psi: Y \to X$ is called an \emph{inverse} of $\phi$ if $\psi(\phi(x)) = x$ for all $x \in X$ and $\phi(\psi(y)) = y$ for all $y \in Y$. In this case, we say that $\phi$ is \emph{invertible} and write $\psi = \phi^{-1}$.

\begin{definition}
A sliding block code $\phi: X \to Y$ is a \emph{conjugacy from $X$ to $Y$} if it is invertible. Two tree-shifts $X$ and $Y$ are called \emph{conjugate}, denoted by $X \cong Y$, if there is a conjugacy from $X$ to $Y$.
\end{definition}

Let $X$ be a tree-shift and let $m$ be a positive integer. We define the \emph{$m$th higher block tree-shift} (or \emph{$m$th higher block presentation}) of $X$, denote by $X^{[m]}$, as follows. Let $\mathcal{A}' = B_m(X)$ be the collection of all allowed $m$-blocks in $X$. Define the $m$th higher block code $\phi_m: X \to (\mathcal{A}')^{\Sigma^*}$ by
\begin{equation}
\phi_m (t)_x = u,
\end{equation}
where $u$ is the $m$-block in $t$ rooted at $x$.

\begin{definition}\label{def:higher-block-tree-shift}
For $m \in \mathbb{N}$, the \emph{$m$th higher block tree-shift} $X^{[m]}$ of a tree-shift $X$ is the image $X^{[m]} = \phi_m (X)$ in the full tree-shift $(\mathcal{A}')^{\Sigma^*}$.
\end{definition}

Notably, $\phi_m$ comes from an $m$-block map and hence is a sliding block code. Let $\psi: X^{[m]} \to X$ be the sliding block code obtained from the one-block map $\Psi: \mathcal{A}' \to \mathcal{A}$ defined as
$$
\Psi(u) = u_{\epsilon},
$$
where $u \in B_m(X)$ is an $m$-block. It is then seen that $\phi_m$ is invertible with inverse $\psi$. This demonstrates the following theorem, which is an extension of a classical result in symbolic dynamics.

\begin{theorem}\label{thm:cong-higher-block}
Tree-shifts $X$ and $X^{[m]}$ are conjugate for all $m \in \mathbb{N}$.
\end{theorem}

\begin{example}\label{eg:even-sum-each-block-part1-1}
Let $X$ be the tree-shift discussed in Example \ref{eg:even-sum-each-block}. Relabeling $B_2(X)$ as $\mathcal{A}' = \{0, 1, 2, 3\}$, where $0, 1, 2, 3$ is
$$
\Tree [.0 0 0 ] \quad \Tree [.0 1 1 ] \quad \Tree [.1 0 1 ] \quad \Tree [.1 1 0 ]
$$
respectively. Theorem \ref{thm:cong-higher-block} infers that $X$ is conjugate to the tree-shift $X' = \phi_2(X)$ over $\mathcal{A}'$. Suppose $t \in X$ is an infinite tree with the pattern in Figure \ref{fig:even-sum-each-block} rooting at $\epsilon$. Then the subtree $\phi_2(t)|_{\Sigma_3}$ of $\phi_2(t)$ is
\begin{center}
\Tree [.2 [.1 3 2 ]
         [.2 1 3 ] ]
\end{center}
\end{example}

A tree-shift $X = \mathsf{X}_{\mathcal{F}}$ is called a \emph{tree-shift of finite type} (TSFT) if the forbidden set $\mathcal{F}$ is finite. Suppose $X$ is a TSFT with forbidden set $\mathcal{F}$. It follows immediately from Theorem \ref{thm:cong-higher-block} that, without loss of generality, we may assume that $u$ is a $2$-block for each $u \in \mathcal{F}$. A TSFT whose forbidden set consists of $2$-blocks is called a \emph{Markov tree-shift}. Readers are referred to \cite{AB-TCS2012} for more details.

\section{Chaotic Tree-shifts}

This section defines the notion of chaos which is suitable for tree-shifts and investigates whether a given tree-shift is chaotic. There are several definitions for chaos, such as Devaney chaos (see \cite{Dev-1987}) and Li-Yorke chaos (see \cite{LY-AMM1975}) for instance, in dynamical systems. In this elucidation, we elucidate the chaotic tree-shifts in the sense of Devaney. We recall Devaney's definition of chaos first.

Let $V$ be a metric space with metric $\mathbf{d}$. $T: V \to V$ is said to be \emph{topologically transitive} if for any pairs of open sets $U_1, U_2 \subset V$ there exists $k > 0$ such that $T^k U_1 \cap U_2 \neq \varnothing$.

Let $x \in V$. We call $x$ a \emph{periodic point} for $T$ if there exists $k > 0$ such that $T^k x = x$; in this case, $\{x, T x, \ldots, T^{k-1} x\}$ is called a \emph{periodic orbit}.

If there exists $\delta > 0$ such that, for any $x \in V$ and any neighborhood $U$ of $x$, there exists $y \in U$ and $n \in \mathbb{N}$ such that $\mathbf{d}(T^n x, T^n y) > \delta$, then we say that $T$ has \emph{sensitive dependence on initial conditions}.

These three notions are the basic ingredients of a chaotic system.

\begin{definition}
A dynamical system $T$ is \emph{chaotic} if:
\begin{enumerate}[\bf (a)]
\item Periodic points for $T$ are dense in $V$.
\item $T$ is topologically transitive.
\item $T$ depends sensitively on initial conditions.
\end{enumerate}
\end{definition}

\subsection{Irreducible Tree-Shifts}

This subsection defines the notions of irreducibility and mixing for tree-shifts. The definitions addressed below are strong. However, these definitions are natural in the way that they extend the theory of shift spaces to tree-shifts.

Let $P \subset \Sigma^*$ be a subset of words. $P$ is called a \emph{prefix set} if no word in $P$ is a prefix of another one. The length of $P$, denoted by $|P|$, is the longest word in $P$. More specifically,
$$
|P| = \left\{
             \begin{array}{ll}
               \max\{|x|: x \in P\}, & \hbox{$P$ is a finite set;} \\
               \infty, & \hbox{otherwise.}
             \end{array}
           \right.
$$
A finite prefix set $P$ is called a \emph{complete prefix set} (CPS) if any $x \in \Sigma^*$, such that $|x| \geq |P|$, has a prefix in $P$.

\begin{definition}\label{def:irreducible}
A tree-shift $X$ is \emph{irreducible} if for each pair of blocks $u, v$ with $u, v \in B_n(X)$, there is a tree $t \in X$ and a complete prefix set $P \subset \bigcup_{k \geq n} \Sigma^k$ such that $u$ is a subtree of $t$ rooted at $\epsilon$ and $v$ is a subtree of $t$ rooted at $x$ for all $x \in P$.
\end{definition}

The definition of irreducible tree-shifts is given in \cite{AB-TCS2013}. It is demonstrated therein that, for any two conjugate tree-shifts $X$ and $Y$, $X$ is irreducible if and only if $Y$ is irreducible.

Suppose $W \subset \Sigma^*$ is a subset of finite words and $x \in \Sigma^*$. We define the concatenation of  $x$ and $W$, denoted by $xW$, as
$$
xW = \{xw: w \in W\}.
$$
Let $t$ be a tree in the tree-shift $X$ and let $u$ be a subtree of $t$ rooted at $x$ for some $x \in \Sigma^*$. For the clarification of discussion, the notation $t|_{xS(u)}$ means the block $u$.

\begin{theorem}\label{thm:equiv-def-irr}
Suppose $X$ is a tree-shift. The following are equivalent.
\begin{enumerate}[\bf (i)]
\item $X$ is irreducible.
\item For each pair of blocks $u \in B_n(X), v \in B_m(X)$, where $n, m \in \mathbb{N}$, there exists $\{P_w\}_{w \in \Sigma^{n-1}}$ with $P_w$ being a CPS for any $w \in \Sigma^{n-1}$ and $t \in X$ such that
$$
t|_{S(u)} = u \quad \text{and} \quad t|_{w x S(v)} = v \text{ for all } w \in \Sigma^{n-1}, x \in P_w.
$$
\item For each pair of blocks $u \in B_n(X), v \in B_m(X)$, where $n, m \in \mathbb{N}$, there exists $\{P_k\}_{1 \leq k \leq l}$ for some $l$ with $P_k$ being a CPS for $1 \leq k \leq l$ and $t \in X$ such that $t|_{S(u)} = u$ and, for each $w \in \Sigma^{n-1}$,
$$
t|_{w x S(v)} = v \text{ for all } x \in P_k \text{ for some } k.
$$
\end{enumerate}
\end{theorem}
\begin{proof}
It suffices to show that (i) is equivalent to (ii) since (ii) being a necessary and sufficient condition for (iii) can be verified in a straightforward manner.

Suppose $X$ is irreducible. For each pair of blocks $u \in B_n(X), v \in B_m(X)$ with $n, m$ being positive integers, a straightforward examination indicates that there is a tree $t \in X$ and a complete prefix set $P \subset \bigcup_{k \geq n} \Sigma^k$ such that
$$
t|_{S(u)} = u \quad \text{and} \quad t|_{xS(v)} = v \text{ for all } x \in P
$$
whether $n = m$ or not. For each $w \in \Sigma^{n-1}$, set
$$
P_w=\{w':ww'\in P\}.
$$
Note that $P_w \neq \varnothing$ for any $w \in \Sigma^{n-1}$ since $P$ is a CPS. Furthermore, $\bigcup_{w \in \Sigma^{n-1}} P_w = P$. It remains to show that $P_w$ is a CPS for each $w$.

Suppose $P_w$ is not a CPS for some $w \in \Sigma^{n-1}$. Without loss of generality, we may assume that $|P_w| + n > |P|$. Then there exists $\alpha = \alpha_1 \ldots \alpha_l$ such that $\alpha_1 \ldots \alpha_k \notin P_w$ for $1\leq k \leq l$ and $l > |P_w|$. This implies that $w \alpha$ does not have a prefix in $P$, which contradicts to the presumption that $P$ is a CPS.

Conversely, suppose $\{P_w\}_{w \in \Sigma^{n-1}}$ is a collection of CPS for a given pair of $n$-blocks $u$ and $v$. Let $P = \bigcup_{w \in \Sigma^{n-1}} P_w$. It suffices to show that $P$ is a CPS. The demonstration can be performed via the analogous argument given above. Hence $X$ is an irreducible tree-shift.

This completes the proof.
\end{proof}

We remark that Theorem \ref{thm:equiv-def-irr} not only clarifies the intuition of the notion of irreducible tree-shifts but indicates the height of blocks $u$ and $v$ in Definition \ref{def:irreducible} need not be the same. Moreover, Theorem \ref{thm:equiv-def-irr} can be used to define the notion of mixing tree-shifts.

\begin{definition}\label{def:mixing}
A tree-shift $X$ is \emph{mixing} if there exist two CPS $P_0$ and $P_1$ such that, for each pair of blocks $u, v \in B(X)$, there is a tree $t \in X$ satisfying
$$
t|_{S(u)}=u \quad \text{and} \quad t|_{wxS(v)} = v\ \text{ for all } x \in P_{w_{n-1}},
$$
where $n = |u|$ and $w = w_1 \ldots w_{n-1} \in \Sigma^{n-1}$.
\end{definition}

\begin{remark}
It follows immediately that a mixing tree-shift is itself irreducible.
\end{remark}

\subsection{Chaotic Tree-Shifts in the Sense of Devaney}

This subsection defines the notion of periodic points of tree-shifts. Note that, instead of only one shift map for a shift space, there are two shift maps, say $\sigma_0$ and $\sigma_1$, for tree-shifts. After defining periodic trees, we then reveal whether a tree-shift is chaotic in the sense of Devaney.

\begin{definition}\label{def:periodic-tree}
Let $X$ be a tree-shift. An infinite tree $t \in X$ is \emph{periodic} if there is a CPS $P$ such that $\sigma_x t=t$ for all $x \in P$, where $\sigma_x =\sigma_{x_k} \circ \sigma_{x_{k-1}} \circ \ldots \circ \sigma_{x_1}$ for $x=x_1 \ldots x_k$.
\end{definition}

\begin{theorem}\label{thm:irr-TSFT-Dense-Periodic-Points}
Suppose a tree-shift of finite type $X$ is irreducible. Then the periodic points of $X$ are dense in $X$.
\end{theorem}
\begin{proof}
To demonstrate that the periodic points are dense in $X$, it suffices to show that for any $t \in  X$ and  $n \in \mathbb{N}$, there is a tree $t' \in X$ which is periodic and $t'|_{\Sigma_{n-1}}=t|_{\Sigma_{n-1}}$, where $\Sigma_m = \bigcup_{k=0}^m \Sigma^k$ is the collection of words whose length are at most $m$.

Suppose $t \in X$ and $n \in \mathbb{N}$ are given. We construct a periodic tree $t'$ as follows.

Let $u = t|_{\Sigma_{n-1}}$ be the $n$-block of $t$ rooted at $\epsilon$. Since $X$ is irreducible, without loss of generality, we may assume that there exist two CPS $P_0$, $P_1$ and a tree $t^1 \in X$ such that
$$
t^1|_{S(u)} = u \quad \text{and} \quad t^1|_{ww'S(u)} = u \text{ for all } w \in \Sigma_{n-1}, w' \in P_{w_{n-1}}.
$$
Set
\begin{align*}
L = \Sigma_{n-1} &\bigcup (\bigcup_{w \in \Sigma^{n-2}} w0P_0S(u)) \bigcup (\bigcup_{w \in \Sigma^{n-2}} w1P_1S(u)) \\
  &\bigcup (\bigcup_{w \in \Sigma^{n-2}} w0P_0^c) \bigcup (\bigcup_{w \in \Sigma^{n-2}} w1P_1^c)
\end{align*}
and $v = t^1|_L$, where $P_i^c = \Sigma_{|P_i|} \setminus P_i$ for $i = 0, 1$. Notably, $v$ is a subtree of $t^1$ but not necessary a block. Without loss of generality, we may assume that $n \geq |\mathcal{F}|$, where $\mathcal{F}$ is the forbidden set of $X$. It comes from $X$ being a TSFT that there exists $t^2 \in X$ such that
$$
t^2|_{S(v)} = v \quad \text{and} \quad (\sigma_{ww'} t^2)|_{S(v)} = v \text{ for all } w \in \Sigma_{n-1}, w' \in P_{w_{n-1}}.
$$
More precisely, it is seen that $\sigma_{ww'} v = u$ for all $w \in \Sigma_{n-1}, w' \in P_{w_{n-1}}$. $X$ being a TSFT with $|\mathcal{F}| \leq n$ infers that the concatenation of $\sigma_{ww'} v$ and $v$ which overlaps the top $n$-block is allowed in $B(X)$.

Analogous to the construction of $t^2$, we construct infinite trees $t^3, t^4, \ldots$ in $X$ which satisfies
$$
t^k|_{S(v)} = v \quad \text{and} \quad (\sigma_{(ww')^i} t^k)|_{S(v)} = v
$$
for all $w \in \Sigma_{n-1}, w' \in P_{w_{n-1}}, k \geq 2$, and $i \leq k-1$. Let $t' = \lim_{k \to \infty} t^k$ and $P = \bigcup_{w \in \Sigma^{n-1}} w P_{w_{n-1}}$. It can be verified that $P$ is a CPS and $\sigma_x t' = t'$ for all $x \in P$. In other words, $t'$ is a periodic tree in $X$.

This completes the proof.
\end{proof}

Theorem \ref{thm:mixing-Dense-Periodic-Points} shows that, beyond Theorem \ref{thm:irr-TSFT-Dense-Periodic-Points}, any mixing tree-shift contains dense periodic trees.

\begin{theorem}\label{thm:mixing-Dense-Periodic-Points}
Suppose $X$ is a mixing tree-shift. Then the periodic points of $X$ are dense in $X$.
\end{theorem}
\begin{proof}
The demonstration is similar to the discussion of Theorem \ref{thm:irr-TSFT-Dense-Periodic-Points}. Hence we only address the main difference in the construction of periodic trees.

Suppose $t \in X$ and $n \in \mathbb{N}$ are given. Since $X$ is mixing, there exists two CPS $P_0, P_1$ that connect any two blocks in $X$. We construct a periodic tree $t'$ as follows.

Let $u = t|_{\Sigma_{n-1}}$ be the $n$-block of $t$ rooted at $\epsilon$ and let $t^1$ be the same one constructed in the proof of Theorem \ref{thm:irr-TSFT-Dense-Periodic-Points}. Set $n_1 = \max\{|P_0|, |P_1|\} + 2n - 1$ and let $v_1 = t^1|_{\Sigma_{n_1-1}}$ be the $n_1$-block of $t^1$ rooted at $\epsilon$.

Suppose $t^k$ is constructed. Let $n_k = k \max\{|P_0|, |P_1|\} + 2n - 1$ and let $v_k = t^k|_{\Sigma_{n_k-1}}$ be the $n_k$-block of $t^k$ rooted at $\epsilon$. It follows from $X$ being mixing that there exists $t^{k+1} \in X$ such that
$$
t^k|_{S(v_k)} = v_k \quad \text{and} \quad t^k|_{ww'S(u)} = u \text{ for all } w \in \Sigma_{n_k-1}, w' \in P_{w_{n_k-1}}.
$$
Let $t' = \lim_{k \to \infty} t^k$. It is seen without difficulty that $t'$ is a periodic tree. This derives the desired result.
\end{proof}

\begin{theorem}\label{thm:irr-Transitive}
If $X$ is an irreducible tree-shift, then $X$ is topologically transitive.
\end{theorem}
\begin{proof}
It suffices to show that $X$ contains a dense orbit.

Let $B(X)$ be an ordered set with lexicographic order. Since $X$ is irreducible, we can construct trees $t^1, t^2, \ldots, t^k, \ldots$ in $X$ such that $u_i$ is a subtree of $t^k$ for $1 \leq i \leq k$. Let $t = \lim_{k \to \infty} t^k$. Then $t$ forms a dense orbit.

This completes the proof.
\end{proof}

A straightforward examination elaborates that every tree-shift is expanding; in other words, any tree-shift depends sensitively on initial conditions. Combining this observation with Theorems \ref{thm:irr-TSFT-Dense-Periodic-Points}, \ref{thm:mixing-Dense-Periodic-Points}, and \ref{thm:irr-Transitive}, we obtain the following corollary.

\begin{corollary}
Suppose $X$ is a tree-shift.
\begin{enumerate}[\bf (a)]
\item If $X$ is an irreducible TSFT, then $X$ is chaotic.
\item If $X$ is mixing, then $X$ is chaotic.
\end{enumerate}
\end{corollary}

\section{Graph Representations of Tree-Shifts of Finite Type}

In this section, we define the graph and matrix representations of a tree-shift of finite type. The elucidation extends the concept of graph and matrix representation of shifts of finite type in classical symbolic dynamics. Furthermore, the concept of graph representations of tree-shifts of finite type motivates possible directions for future research.

A tree-shift of finite type $X$ is a \emph{$k$-height tree-shift of finite type} if its forbidden set consists of $(k+1)$-blocks. Recall that a one-height TSFT is called a Markov tree-shift. Theorem \ref{thm:cong-higher-block} suggests that, without loss of generality, we may assume every TSFT $X$ is a Markov tree-shift.

A graph $G$ consists of a finite set $\mathcal{V}$ of vertices (or states) together with a finite set $\mathcal{E}$ of edges. Each edge $e \in \mathcal{E}$ starts at a vertex denoted by $\mathrm{i}(e) \in \mathcal{V}$ and terminates at a vertex denoted by $\mathrm{t}(e) \in \mathcal{V}$. Equivalently, the edge $e$ has an initial state $\mathrm{i}(e)$ and a terminal state $\mathrm{t}(e)$. An alternative expression of the edge $e$ is $e = (a, b)$, where $a \in \mathcal{V}$ is the initial state of $e$ and $b \in \mathcal{V}$ is the terminal state.

A graph is called \emph{essential} if there is no stranded vertex. More precisely, each vertex in an essential graph is an initial state of one edge and is a terminal state of another edge.

The adjacency matrix $A$ of graph $G$ is defined as $A = (A(I, J))$, where $I, J \in \mathcal{V}$ are the vertices in $G$ and $A(I, J)$ denotes the number of edges in $G$ with initial state $I$ and terminal state $J$. Without loss of generality, we assume that there is at most one edge with initial state $I$ and terminal state $J$ for all $I, J$. Suppose $G_0$ and $G_1$ are two graphs. The disjoint union of $G_0$ and $G_1$, denoted by $G = G_0 \bigsqcup G_1$, is the graph consisting of vertex set $\mathcal{V}(G) = \mathcal{V}(G_0) \bigcup \mathcal{V}(G_1)$ together with edge set $\mathcal{E}(G) = \mathcal{E}(G_0) \bigcup \mathcal{E}(G_1)$. In other words, $G$ consists of two separate graphs.

\begin{definition}\label{def:matrix-shift}
Let $G = G_0 \bigsqcup G_1$ be the disjoint union of graphs $G_0$ and $G_1$, and let $A_0$ and $A_1$ be the adjacency matrix of $G_0$ and $G_1$, respectively. The \emph{vertex tree-shift} $\mathsf{X}_G$ is the tree-shift over the alphabet $\mathcal{A} = \{0, 1, \ldots, m-1\}$, where $m = \max\{|\mathcal{V}(G_0)|, |\mathcal{V}(G_1)|\}$ and $A_i$ is indexed by $\{0, 1, \ldots, |\mathcal{V}(G_i)|-1\}$ for $i = 0, 1$, specified by
\begin{equation}
\mathsf{X}_G = \{t \in \mathcal{A}^{\Sigma^*}: A_0(t_x, t_{x0}) = 1 \text{ and } A_1(t_x, t_{x1}) = 1 \text{ for all } x \in \Sigma^*\}.
\end{equation}
The shift maps on $\mathsf{X}_G$ are $\sigma_0$ and $\sigma_1$ as defined in Section 2.
\end{definition}

The following proposition comes immediately from Theorem \ref{thm:cong-higher-block} and Definition \ref{def:matrix-shift}. Hence the proof is omitted.

\begin{proposition}\label{prop:TSFT-is-vertex-shift-is-Markov}
Every vertex tree-shift is a Markov tree-shift. Conversely, every tree-shift of finite type is conjugate to a vertex tree-shift.
\end{proposition}

For the rest of this paper, a tree-shift of finite type is referred as a vertex tree-shift with essential graph representation unless otherwise stated. Suppose $x = x_1 x_2 \ldots x_l \in \Sigma^*$. Define $A_x = A_{x_l} A_{x_{l-1}} \cdots A_{x_1}$ as the product of $A_0$ and $A_1$; increase the size of matrix if necessary. A necessary and sufficient condition for tree-shift $X$ being irreducible then follows.

\begin{theorem}\label{thm:iff-cond-and-onlyif-cond-irr-mixing}
Suppose $X$ is a TSFT with graph representation $G = G_0 \bigsqcup G_1$ and adjacency matrices $A_0$ and $A_1$.
\begin{enumerate}[\bf (a)]
\item If $X$ is irreducible, then $A_0$ and $A_1$ are both of the same size and irreducible. In other words, $|\mathcal{V}(G_0)| = |\mathcal{V}(G_1)|$.
\item $X$ is irreducible if and only if for each pair $i, j \in \mathcal{A}$ there exists a CPS $P$ such that $A_x(i, j) > 0$ for all $x \in P$.
\item $X$ is mixing if and only if there exists a CPS $P$ such that $A_x(i, j) > 0$ for all $x \in P$ and $i, j \in \mathcal{A}$.
\end{enumerate}
\end{theorem}
\begin{proof}
(a) Without loss of generality, we assume that $A_0$ is an $m \times m$ matrix and $A_1$ is a $(m+1) \times (m+1)$ matrix. Consider two one-blocks $0$ and $m$, it is easily seen that there is no tree $t \in X$ and $k \in \mathbb{N}$ such that $t_{\epsilon} = 0$ and $t_{0^k} = m$. In other words, there exists no CPS $P$ and tree $t$ such that $t_{\epsilon} = 0$ and $t_x = m$ for all $x \in P$. This contradicts the presumption that $X$ is irreducible.

Similar argument demonstrates that $X$ is not irreducible if either $A_0$ or $A_1$ is not irreducible.

(b) To see that $X$ is irreducible if and only if for each pair $i, j \in \mathcal{A}$ there exists a CPS $P$ such that $A_x(i, j) > 0$ for all $x \in P$, it suffices to show the ``if'' part.

Given a pair $u, v \in B(X)$, for any $w \in \Sigma^{|u|-1}$, there exists a CPS $P_w$ such that $A_x(u_w, v_{\epsilon}) > 0$ for all $x \in P_w$. Therefore, there exists a tree $t^{(w)} \in X$ with $t^{(w)}_{\epsilon} = u_w$ and $t^{(w)}_x = v_{\epsilon}$ for all $x \in P_w$. Moreover, we can choose $t^{(w)}$ which satisfies $t^{(w)}|_{xS(v)} = v$ for all $x \in P_w$. Let $M = \max\{|P_w|: w \in \Sigma^{|u|-1}\}$, and let $v^{(w)} = t^{(w)}|_{\Sigma_{M+|v|-1}}$. $X$ is a Markov tree-shift indicates that there is a tree $t \in X$ such that $t|_{S(u)} = u$ and $t|_{wS(v^{(w)})} = v^{(w)}$ for all $w \in \Sigma^{|u|-1}$.

Set $P = \bigcup_{w \in \Sigma^{|u|-1}} w P_w$. It can be verified that $P$ is a CPS, $t|_{S(u)} = u$, and $t|_{xS(v)} = v$ for all $x \in P$. This demonstrates that $X$ is irreducible and completes the proof.

(c) The discussion is similar to the elaboration in (b), hence is omitted.
\end{proof}

Suppose $A, B \in \mathcal{M}_n(\mathbb{R})$ are $n \times n$ matrices. We say that $A \leq B$ if $A(i, j) \leq B(i, j)$ for $1 \leq i, j \leq n$. Furthermore, $A < B$ means that $A \leq B$ and $A(i, j) < B(i, j)$ for some $i, j$. A careful but routine verification elaborates the following corollary, thus the proof is omitted.

\begin{corollary}
Suppose $X$ is a TSFT with adjacency matrices $A_0$ and $A_1$.
\begin{enumerate}[\bf (a)]
\item If $A_0 = A_1 = A$, then $X$ is irreducible if and only if $A$ is irreducible.
\item If $A_0 = A_1 = A$, then $X$ is mixing if and only if $A$ is primitive.
\item Suppose $Y$ is an irreducible tree-shift with adjacency matrices $B_0$ and $B_1$. If $A_0 \geq B_{i_1}$ and $A_1 \geq B_{i_2}$ for some $0 \leq i_1, i_2 \leq 1$, then $X$ is irreducible.
\item If $A_0$ (resp.~$A_1$) is irreducible and $A_1 \geq A_0$ (resp.~$A_0 \geq A_1$), then $X$ is irreducible.
\end{enumerate}
\end{corollary}

The following example elucidates a mixing tree-shift with two distinct adjacency matrices.

\begin{example}
Suppose $X$ is a TSFT with adjacency matrices
$$
A_0 = \begin{pmatrix}
1 & 1\\
1 & 1
\end{pmatrix}
\quad \text{and} \quad
A_1 = \begin{pmatrix}
1 & 1\\
1 & 0
\end{pmatrix}.
$$
It can be seen that $A_0$ and $A_1$ are both primitive matrices. Moreover,
$
A_1 A_0 = \begin{pmatrix}
2 & 2 \\
1 & 1
\end{pmatrix}
$
is a positive matrix. Let $P = \{0, 10, 11\}$. Then $P$ is a CPS and Theorem \ref{thm:iff-cond-and-onlyif-cond-irr-mixing} (c) infers that $X$ is mixing.
\end{example}

\begin{proposition}\label{prop:2symbol-is-irr}
Suppose $X$ is a TSFT with $2 \times 2$ adjacency matrices $A_0$ and $A_1$. If $A_0$ and $A_1$ are both irreducible and $A_0 A_1 = A_1 A_0$, then $X$ is irreducible.
\end{proposition}
\begin{proof}
Theorem \ref{thm:iff-cond-and-onlyif-cond-irr-mixing} mentions that, to show that $X$ is irreducible, it suffices to elaborate for every $i, j \in \mathcal{A}$ there exists a CPS $P$ such that $A_x(i, j) > 0$ for all $x \in P$. For a given pair $i, j \in \{0, 1\}$, we divide the proof into $4$ cases.

\noindent \textbf{Case 1.} $A_0(i, j) > 0$ and $A_1(i, j) > 0$. The desired CPS is $P=\{0,1\}$.

\noindent \textbf{Case 2.} $A_0(i, j) = 0$ and $A_1(i, j) > 0$. Notably, $A_0$ is a $2 \times 2$ irreducible matrix infers that $A_0^2(i, j) > 0$. \\
\textbf{Case 2a.}  If $(A_0 A_1)(i, j) > 0$, then $P = \{00, 01, 1\}$ is the desired CPS. \\
\textbf{Case 2b.}  Suppose $(A_0 A_1)(i, j) = 0$. It is seen then
$$
A_0(i, j) A_1(j, j) + A_0(i, \hat{j}) A_1(\hat{j}, j)=0,
$$
where $\hat{j} + j = 1$. Since $A_0^2(i, j) > 0$ and $A_1$ is irreducible, we have
$$
(A_0^2 A_1)(i, j) \geq A_0^2(i, j) A_1(j, j) > 0.
$$
The commute of $A_0$ and $A_1$ implies $(A_0 A_1 A_0)(i, j) > 0$. Furthermore,
\begin{align*}
(A_0 A_1^2)(i, j) &= A_0(i, j) A_1^2(j, j) + A_0(i, \hat{j}) A_1^2(\hat{j}, j) \\
  &\geq A_0(i, \hat{j}) A_1^2(\hat{j}, j)>0.
\end{align*}
In this case, a CPS $P$ is considered as $P = \{00, 010, 011, 1\}$.

\noindent \textbf{Case 3.} $A_0(i, j) > 0$ and $A_1(i, j) = 0$. Analogous to the discussion in Case 2, it comes that $P = \{0, 10, 11\}$ if $(A_1 A_0)(i, j) > 0$, and $P = \{0, 11, 100, 101\}$ if $(A_1 A_0)(i, j) = 0$.

\noindent \textbf{Case 4.} $A_0(i, j) = A_1(i, j) = 0$. Since $A_0$ and $A_1$ are irreducible, it follows that $A_0^2(i, j) > 0$, $A_1^2(i, j) > 0$, $A_0(i, \hat{j}) > 0$, and $A_1(\hat{i}, j) > 0$. \\
\textbf{Case 4a.} If $(A_0 A_1)(i, j) > 0$, then $P = \{00, 01, 10, 11\}$ is the desired CPS. \\
\textbf{Case 4b.} Suppose $(A_0 A_1)(i, j) = 0$. More specifically,
$$
A_0(i, j) A_1(j, j) + A_0(i, \hat{j}) A_1(\hat{j}, j) = 0.
$$
Irreducibility of $A_0$ and $A_1$ asserts that
$$
(A_0 A_1^2)(i, j) \geq A_0(i, \hat{j}) A_1^2(\hat{j}, j) > 0
$$
and
$$
(A_0^2 A_1)(i, j) \geq A_0^2(i, j) A_1(j, j) > 0.
$$
The commute of $A_0$ and $A_1$ implies
$$
(A_0 A_1 A_0)(i, j) > 0, (A_1 A_0^2)(i, j) > 0, \text{ and } (A_1 A_0 A_1)(i, j) > 0.
$$
This suggests a CPS is then given by $P = \{00, 11, 010, 011, 100, 101\}$.

Cases $1$ - $4$ conclude that $X$ is irreducible. This completes the proof.
\end{proof}

To conclude this section, we introduce a necessary and sufficient condition for determining whether a TSFT $X$ is irreducible, which extends the corresponding result in symbolic dynamics.

A \emph{labeled graph} $\mathcal{G}$ is a pair $(G, \mathcal{L})$, where $G$ is a graph with edge set $\mathcal{E}$, and the \emph{labeling} $\mathcal{L}: \mathcal{E} \to \Upsilon$ assigns each edge $e$ of $G$ a label $\mathcal{L}(e)$ from the finite alphabet $\Upsilon$. The underlying graph of $\mathcal{G}$ is $G$.

Just as a graph $G$ is conveniently described by its adjacency matrix $A$, a labeled graph $\mathcal{G}$ has an analogous \emph{symbolic adjacency matrix} $S$. The entry $S(I, J)$ is the formal sum of the labels of all edges from $I$ to $J$, or a ``zero'' character $\varnothing$ if there are no such edges.

Suppose $X$ is a TSFT with graph representation $G = G_0 \bigsqcup G_1$ and adjacency matrices $A_0$ and $A_1$. The labeled graph representation $\mathcal{G}$ is defined as follows. The underlying graph $\overline{G}$ of $\mathcal{G}$ is the union of $G_0$ and $G_1$. More precisely, $\overline{G}$ has vertex set $\mathcal{V}(\overline{G}) = \mathcal{V}(G_0) = \mathcal{V}(G_1)$ and edge set $\mathcal{E}(\overline{G}) = \mathcal{E}(G_0) \bigcup \mathcal{E}(G_1)$. The labeling $\mathcal{L}: \mathcal{E}(\overline{G}) \longrightarrow \Sigma = \{0, 1\}$ is defined as
\begin{equation}
\mathcal{L}(e) = \left\{
  \begin{array}{ll}
  0, & \hbox{if $e \in \mathcal{E}(G_0)$;} \\
  1, & \hbox{if $e \in \mathcal{E}(G_1)$;}
  \end{array}
  \right.
\end{equation}

For $\pi = e_1 e_2 \ldots e_k$ a path in $\overline{G}$, i.e., $\mathrm{t}(e_i) = \mathrm{i}(e_{i+1})$ for $i = 1, \ldots, k-1$, let
$$
\mathcal{L}(\pi) = \mathcal{L}(e_1) \mathcal{L}(e_2) \ldots \mathcal{L}(e_k)
$$
be the simplification of notation. We remark that an alternative expression of a path $\pi = e_1 e_2 \ldots e_k$ is $\pi = v_0 v_1 \ldots v_k$, where $v_i \in \mathcal{V}(\overline{G})$ and $e_{i+1} = (v_i, v_{i+1})$. Suppose $w = w_1 w_2 \ldots w_k \in \Sigma^k$ is a labeled path in $\mathcal{G}$. The collection of underlying paths of $w$ in $\overline{G}$ is then described as $\mathcal{L}^{-1}(w)$. Furthermore, the collection of terminal states of $w$ is
\begin{equation}\label{eq:possible-terminal-state}
V_w = \{\mathrm{t}(\pi): \pi \in \mathcal{L}^{-1}(w)\}.
\end{equation}

\begin{definition}
A labeled path $w = w_1 w_2 \ldots w_k$ in $\mathcal{G}$ is a \emph{cycle} if $V_w = V_{w_1}$ and $w_k = w_1$.
\end{definition}

Suppose $x \in \Sigma^*$ is a finite word. A prefix $x'$ of $x$ is denoted by $x' \preceq x$. The following lemma illustrates that an irreducible TSFT contains no cycle $w \in \Sigma^*$ and $i, j \in \mathcal{A}$ such that $A_{w'}(i, j) = 0$ for all $w' \preceq w$.

\begin{lemma}\label{lem:cycle-non-irr}
Suppose $X$ is an irreducible TSFT with graph representation $G = G_0 \bigsqcup G_1$ and adjacency matrices $A_0$ and $A_1$. Then, for each pair $i, j \in \mathcal{A}$, there is no cycle $w \in \Sigma^*$ such that $A_{w'}(i, j) = 0$ for all prefix $w'$ of $w$.
\end{lemma}
\begin{proof}
Suppose there is a cycle $w \in \Sigma^m$ such that $A_{w'}(i, j) = 0$ for all $w' \preceq w$ and for some $i, j$. We claim that, for any CPS $P$, there is $x \in P$ such that $A_x(i, j) = 0$. This derives a contradiction since $X$ is irreducible.

Obviously, if there exists $x \in P$ such that $x \preceq w$, then $A_x(i, j) = 0$.

Suppose $x \not \preceq w$ for all $x \in P$. $P$ being a CPS asserts that there exists a CPS $P_1$ such that $\hat{w} P_1 \subset P$, where $\hat{w} = w_1 w_2 \ldots w_{m-1}$. If there is $x \in P_1$ such that $x \preceq w$, then $A_{\hat{w}x}(i, j) = A_x(i, j) = 0$ since $w$ is a cycle. Otherwise, analogous to the earlier discussion, there exists a CPS $P_2$ such that $\hat{w} P_2 \subset P_1$.

Repeating the above processes, it is seen that there is a unique $k \in \mathbb{N}$ satisfying $k |w| \leq |P| < (k+1) |w|$. Hence there exists $x \in P_{k+1}$ such that $\hat{w}^k x \in P$ and $A_{\hat{w}^k x}(i, j) = 0$. The proof is complete.
\end{proof}

Recall that every TSFT $X$ over a finite alphabet $\mathcal{A}$ is associated with a graph representation $G = G_0 \bigsqcup G_1$ and a labeled graph representation $\mathcal{G} = (\overline{G}, \mathcal{L})$, herein $\overline{G}$ is obtained from merging $G_0$ with $G_1$. More specifically, the vertex set of $\overline{G}$ is, up to the change of symbols, the alphabet $\mathcal{A}$. Theorem \ref{thm:matrix-cycle-not-irr} elaborates that the verification of irreducibility of $X$ can be done in finite steps.

\begin{theorem}\label{thm:matrix-cycle-not-irr}
Suppose $X$ is a TSFT with graph representation $G = G_0 \bigsqcup G_1$ and $n \times n$ adjacency matrices $A_0$ and $A_1$.  If there exists $w \in \Sigma^{n2^{n-1}}$ and $i, j \in \mathcal{A}$ such that $A_{w'}(i, j) = 0$ for all $w' \preceq w$, then $X$ is not irreducible.
\end{theorem}
\begin{proof}
Note that $A_0, A_1$ being $n \times n$ matrices infers that $|\mathcal{A}| = n$. We claim that each word $w$ of length $n 2^{n-1}$ contains a cycle. Lemma \ref{lem:cycle-non-irr} asserts that $X$ is not irreducible.

Without loss of generality, we assume that $w_{n 2^{n-1}} = 0$. The irreducibility of $A_0$ and $A_1$ imply that there exists $1 \leq k_1 \leq n$ such that $w_{k_1} = 0$. Similarly, there exists $k_l +1 \leq k_{l+1} \leq k_l + n$ such that $w_{k_l} = 0$ for $l \geq 1$ and $k_{l+1} \leq n 2^{n-1}$. It follows that $w_1 w_2 \ldots w_{n2^{n-1}-1}$ contains at least $(2^{n-1}-1)$ $0$'s. Moreover, the possible choices of terminal states of $\mathcal{L}^{-1}(0)$ is $2^{n-1} - 1$. The Pigeonhole Principle ensures $V_{w_1 \ldots w_{k_l}} = V_w$ for some $l$, where $V_{\omega}$ is defined in \eqref{eq:possible-terminal-state} for $\omega \in \Sigma^*$. That is, $w$ contains a cycle.

This completes the proof.
\end{proof}

Suppose $X$ is a TSFT with labeled graph representation $\mathcal{G} = (\overline{G}, \mathcal{L})$ and adjacency matrices $A_0$ and $A_1$. Recall that $A_0, A_1$ are both $0$-$1$ matrices. For $k = 0, 1$, define the symbolic adjacency matrix $S_k$ as
\begin{equation}
S_k(i, j) = \left\{
  \begin{array}{ll}
  k, & A_k(i, j) = 1\hbox{;} \\
  \varnothing, & \hbox{otherwise.}
  \end{array}\right.
\end{equation}
The symbolic adjacency matrix $S$ of $\mathcal{G}$ is defined as $S = S_0 + S_1$, where the summation is the formal sum of the labels. Notably, every entry in $S^k$ is the summation of words in $\Sigma^*$. Combining the symbolic adjacency matrix $S$ together with Theorem \ref{thm:matrix-cycle-not-irr} signifies Corollary \ref{cor:symbolic-verify-irr}. The proof is straightforward, and hence is omitted.

\begin{corollary}\label{cor:symbolic-verify-irr}
Let $X$ be a TSFT with $n \times n$ symbolic adjacency matrix $S$. Let $\mathbf{S} = \Sigma_{k=1}^{n 2^{n-1}} S^k$. Then $X$ is irreducible if and only if, for each pair $i, j \in \mathcal{A}$, $\mathbf{S}(i, j)$ contains the formal sum of entries of $P$ for some CPS $P$.
\end{corollary}

\begin{definition}\label{def:irreducible-symbolic-matrix}
An $n \times n$ symbolic matrix $\mathbf{M}$ over an finite alphabet $\Lambda$ is called \emph{irreducible} if, for $1 \leq i, j \leq n$, there exists $l \in \mathbb{N}$ such that $\mathbf{M}^l(i, j)$ contains the formal sum of a CPS $P$ in $\Lambda^*$.
\end{definition}

Combining Corollary \ref{cor:symbolic-verify-irr} together with Definition \ref{def:irreducible-symbolic-matrix} we have extended the theorem which verifies the irreducibility of shifts of finite type in symbolic dynamics. Namely,
\begin{quote}
A TSFT $X$ is irreducible if and only if the corresponding symbolic adjacency matrix $S$ is irreducible. Moreover, an $n \times n$ symbolic adjacency matrix $S$ is irreducible if and only if, for each $i, j$, $S^{k_{i,j}}(i, j)$ contains the formal sum of a CPS with $k_{i,j} \leq n 2^{n-1}$.
\end{quote}

\begin{example}\label{eg:TSFT-irreducible}
Suppose $X$ is a TSFT with adjacency matrices
$$
A_0 = \begin{pmatrix}
1 & 1 \\
1 & 0
\end{pmatrix}
\quad \text{and} \quad
A_1 = \begin{pmatrix}
0 & 1 \\
1 & 1
\end{pmatrix}.
$$
Theorem \ref{thm:iff-cond-and-onlyif-cond-irr-mixing} (or Proposition \ref{prop:2symbol-is-irr}) infers that $X$ is an irreducible TSFT.

It is seen that the symbolic adjacency matrix of $X$ is
$$
S = \begin{pmatrix}
0 & 0 + 1 \\
0 + 1 & 1
\end{pmatrix}.
$$
See Figure \ref{fig:label-graph-for-irr-TSFT} for the labeled graph representation of $X$. Note that $0, 1$ in $S$ are symbols rather than integers. Moreover,
$$
S^2 = \begin{pmatrix}
2 \cdot 00 + 01 + 10 + 11 & 00 + 2 \cdot 01 + 11 \\
00 + 2 \cdot 10 + 11 & 00 + 01 + 10 + 2 \cdot 11
\end{pmatrix}
$$
and
$$
S + S^2 = \begin{pmatrix}
0 + 2 \cdot 00 + 01 + 10 + 11 & 0 + 1 + 00 + 2 \cdot 01 + 11 \\
0 + 1 + 00 + 2 \cdot 10 + 11 & 1 + 00 + 01 + 10 + 2 \cdot 11
\end{pmatrix}.
$$
This shows that, for each $i, j \in \{0, 1\}$, a CPS $P_{i, j}$ contained in $S + S^2$ is
\begin{align*}
P_{0, 0} &= \{0, 10, 11\}, & P_{0, 1} &= \{0, 1\}, \\
P_{1, 0} &= \{0, 1\}, & P_{1, 1} &= \{00, 01, 1\}.
\end{align*} 
It also concludes that $X$ is irreducible.

\begin{figure}
\begin{center}
\psset{unit=0.8cm}
\begin{pspicture}(8,2)
\psset{nodesep=0.1cm}
\rput(2,1){\ovalnode{A}{$0$}}  \rput(6,1){\ovalnode{B}{$1$}}

\ncarc[arcangle=20]{->}{A}{B}\Aput{$0 + 1$}  \ncarc[arcangle=20]{->}{B}{A}\Aput{$0 + 1$}  \nccurve[angleA=30,angleB=-30,ncurv=5.5]{->}{B}{B}\Aput{$1$} \nccurve[angleA=150,angleB=210,ncurv=5.5]{->}{A}{A}\Bput{$0$}
\end{pspicture}
\end{center}
\caption{Labeled graph representation for the irreducible tree-shift of finite type investigated in Example \ref{eg:TSFT-irreducible}.}
\label{fig:label-graph-for-irr-TSFT}
\end{figure}
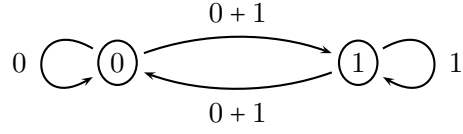
\end{example}

\begin{remark}
It is remarkable that the choice of CPS is not unique, which is seen from the above example.
\end{remark}

An extension of Corollary \ref{cor:symbolic-verify-irr} reveals a necessary and sufficient condition for verifying whether a TSFT is mixing.

\begin{theorem}\label{thm:primitive-matrix-finite-steps}
Let $X$ be a TSFT with $n \times n$ symbolic adjacency matrix $S$. Then $X$ is mixing if and only if there exists an integer $k \leq n^3 2^{2(n-1)}$ such that $S^k(i, j)$ contains the formal sum of entries of a CPS $P$ for $1 \leq i, j \leq n$.
\end{theorem}

\begin{definition}\label{def:primitive-symbolic-matrix}
An $n \times n$ symbolic matrix $\mathbf{M}$ over an finite alphabet $\Lambda$ is called \emph{primitive} if there exists $l \in \mathbb{N}$ such that $\mathbf{M}^l(i, j)$ contains the formal sum of a CPS $P$ in $\Lambda^*$ for $1 \leq i, j \leq n$.
\end{definition}

Combining Theorem \ref{thm:primitive-matrix-finite-steps} together with Definition \ref{def:primitive-symbolic-matrix} we have extended the theorem which verifies the mixing of shifts of finite type in symbolic dynamics. Namely,
\begin{quote}
A TSFT $X$ is mixing if and only if the corresponding symbolic adjacency matrix $S$ is primitive. Moreover, an $n \times n$ symbolic adjacency matrix $S$ is primitive if and only if $S^k(i, j)$ contains the formal sum of a CPS with $k \leq n^3 2^{2(n-1)}$ for each $i, j$.
\end{quote}

\begin{proof}[Proof of Theorem \ref{thm:primitive-matrix-finite-steps}]
It suffices to show that, for a primitive matrix $S$, $S^k(i, j)$ contains the formal sum of a CPS with $k \leq n^3 2^{2(n-1)}$ for each $i, j$.

Notably, a primitive symbolic matrix is itself irreducible. Since $S$ is primitive, there exists $k_1 \in \mathbb{N}$ such that $S^{k_1}(1, 1)$ contains the formal sum of entries of a CPS. It is easily seen that $S^l$ is irreducible for all $l \in \mathbb{N}$. Hence, there exist $k_2, k_3, \ldots, k_n \in \mathbb{N}$ such that, for $l \leq l \leq n$, $S^{k_1 \cdots k_l}(i, i)$ contains the formal sum of entries of a CPS for $1 \leq i \leq l$. In other words, each diagonal entry of $S^{k_1  \cdots k_n}$ contains the formal sum of a CPS.

Write $S^{k_1  \cdots k_n} = D + \overline{S}$, where $D$ is an $n \times n$ diagonal matrix such that each entry of $D$ is either empty or the formal sum of a CPS. It can be seen that $\overline{S}$ is irreducible. Corollary \ref{cor:symbolic-verify-irr} indicates that every entry of $\Sigma_{i=1}^{n 2^{n-1}}\overline{S}^i$ contains a CPS. It follows that
$$
(S^{k_1 \cdots k_n})^{n 2^{n-1}} = (D + \overline{S})^{n 2^{n-1}} \geq \sum_{i=1}^{n 2^{n-1}} \overline{S}^i,
$$
where two symbolic matrices $\mathbf{M}, \mathbf{N}$ are denoted by $\mathbf{M} \geq \mathbf{N}$ if $\mathbf{M}(i, j)$ contains $\mathbf{N}(i, j)$ for all $i, j$. Therefore, every entry of $(S^{k_1 \cdots k_n})^{n 2^{n-1}}$ contains a CPS.

Furthermore, Theorem \ref{thm:matrix-cycle-not-irr} asserts that $k_i \leq n 2^{n-1}$ for $1 \leq i \leq n$. Hence,
$$
(S^{k_1 \cdots k_n})^{n 2^{n-1}} \leq (S^{n 2^{n-1}})^{n \cdot (n 2^{n-1})} = S^{n^3 2^{2(n-1)}}.
$$
This completes the proof.
\end{proof}

\begin{example}\label{eg:even-sum-each-block-part2}
The tree-shift of finite type $X$ discussed in Example \ref{eg:even-sum-each-block}, by Theorem \ref{thm:cong-higher-block}, is conjugate to the vertex tree-shift $X'$ over $\mathcal{A}' = \{0, 1, 2, 3\}$, where $0, 1, 2, 3$ is
$$
\Tree [.0 0 0 ] \quad \Tree [.0 1 1 ] \quad \Tree [.1 0 1 ] \quad \Tree [.1 1 0 ]
$$
respectively, with adjacency matrices
$$
A_0 = \begin{pmatrix}
1 & 1 & 0 & 0 \\
0 & 0 & 1 & 1 \\
1 & 1 & 0 & 0 \\
0 & 0 & 1 & 1
\end{pmatrix}
\quad \text{and} \quad
A_1 = \begin{pmatrix}
1 & 1 & 0 & 0 \\
0 & 0 & 1 & 1 \\
0 & 0 & 1 & 1 \\
1 & 1 & 0 & 0
\end{pmatrix}.
$$
See Figure \ref{fig:label-graph-for-even-sum-eg} for the labeled graph representation of $X'$. The symbolic adjacency matrix of $X'$ is
$$
S = \begin{pmatrix}
0 + 1 & 0 + 1 & \varnothing & \varnothing \\
\varnothing & \varnothing & 0 + 1 & 0 + 1 \\
0 & 0 & 1 & 1 \\
1 & 1 & 0 & 0
\end{pmatrix}.
$$
It follows from
$$
S^2 = \begin{pmatrix}
(0 + 1)^2 & (0 + 1)^2 & (0 + 1)^2 & (0 + 1)^2 \\
(0 + 1)^2 & (0 + 1)^2 & (0 + 1)^2 & (0 + 1)^2 \\
(0 + 1)^2 & (0 + 1)^2 & (0 + 1)^2 & (0 + 1)^2 \\
(0 + 1)^2 & (0 + 1)^2 & (0 + 1)^2 & (0 + 1)^2
\end{pmatrix}
$$
that the any two blocks can be connected through the CPS $P = \{00, 01, 10, 11\}$. This demonstrates that $X'$ is mixing, and so is $X$.

\begin{figure}
\begin{center}
\psset{unit=0.8cm}
\begin{pspicture}(9,8)
\psset{nodesep=0.1cm}
\rput(2.5,1){\ovalnode{A}{$0$}}  \rput(7,1){\ovalnode{D}{$3$}}
\rput(2.5,6){\ovalnode{B}{$1$}}  \rput(7,6){\ovalnode{C}{$2$}}

\ncarc[arcangle=20]{->}{B}{C}\Aput{$0, 1$}  \ncarc[arcangle=20]{->}{C}{B}\Aput{$0$} \ncarc[arcangle=20]{->}{B}{D}\Aput{$0, 1$}  \ncarc[arcangle=20]{->}{D}{B}\Aput{$1$} \ncarc[arcangle=20]{->}{C}{D}\Aput{$1$}  \ncarc[arcangle=20]{->}{D}{C}\Aput{$0$} \ncarc[arcangle=20]{->}{C}{A}\Aput{$0$}
\nccurve[angleA=30,angleB=-30,ncurv=5.5]{->}{D}{D}\Aput{$0$}
\nccurve[angleA=30,angleB=-30,ncurv=5.5]{->}{C}{C}\Aput{$1$}
\nccurve[angleA=150,angleB=210,ncurv=5.5]{->}{A}{A}\Bput{$0, 1$}
\ncline{->}{A}{B}\Aput{$0, 1$} \ncline{->}{D}{A}\Aput{$1$}
\end{pspicture}
\end{center}
\caption{Labeled graph representation for the mixing tree-shift of finite type elaborated in Example \ref{eg:even-sum-each-block-part2}, which is conjugate to the tree-shift given in Example \ref{eg:even-sum-each-block}.}
\label{fig:label-graph-for-even-sum-eg}
\end{figure}
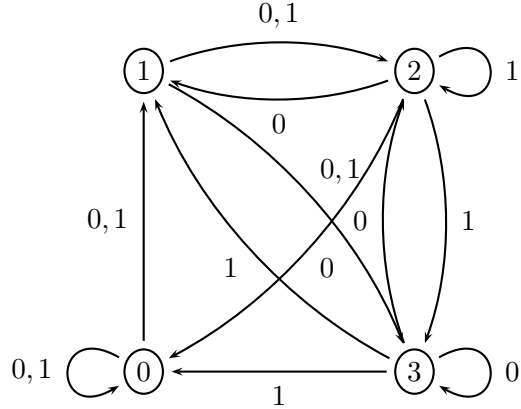
\end{example}

\section{Conclusion and Discussion}

In this article, we define the notions of chaos (in the sense of Devaney) for tree-shifts and show that both mixing tree-shifts and irreducible tree-shifts of finite type are chaotic. Furthermore, the graph and labeled graph representations of tree-shifts of finite type are given. The necessary and sufficient condition for determining whether a tree-shift of finite type is irreducible is revealed. More precisely, with the introduction of the adjacency matrix and symbolic adjacency matrix, the irreducibility and mixing of $X$ is verified via matrix operations. Most importantly, the verification can be done in finite steps with an upper bound.

This work extends the theory of irreducibility and the mixing of shifts of finite type and matrices in symbolic dynamics. Further investigations such as the irreducibility of sofic tree-shifts are in preparation. In the mean time, the following open problems remain of interest.

A tree-shift $X$ is called a \emph{specification tree-shift} if there exists finite complete prefix sets $P_1, P_2, \ldots, P_{\ell}$ such that, for any two $u$ and $v$, there is a tree $t \in X$ satisfying the property: For each $w \in \Sigma^{|u|-1}$, there exists $1 \leq i \leq \ell$ such that
$$
t|_{S(u)} = u \quad \text{and} \quad t_{w x S(v)} = v \text{ for all } x \in P_i.
$$
It is easily seen that
\begin{center}
$X$ is mixing \quad $\Longrightarrow$ \quad $X$ is specification \quad $\Longrightarrow$ \quad $X$ is irreducible.
\end{center}
The following problems then arose.

\begin{problem}
Suppose $X$ is a specification tree-shift. Is $X$ chaotic? Moreover, is $X$ chaotic if $X$ is a specification tree-shift of finite type?
\end{problem}

\begin{problem}
Suppose $X$ is a tree-shift of finite type with the symbolic adjacency matrix $S$. How do we determine whether $X$ is specific via $S$?
\end{problem}

Let $|B_m(X)|$ denote the number of $m$-blocks in $X$. The \emph{topological entropy} of $X$ is defined as
$$
h(X) = \lim_{m \to \infty} \frac{\ln^2 |B_m(X)|}{m}
$$
whenever the limit exists, where $\ln^2 = \ln \circ \ln$.

\begin{problem}
Suppose $X$ is a tree-shift. Does $h(X) > 0$ imply the chaos of $X$?
\end{problem}

\begin{problem}
Suppose $X$ is an irreducible tree-shift of finite type. Is $h(X) > 0$? Does $X$ being a mixing tree-shift imply $h(X) > 0$?
\end{problem}

The topological entropy of tree-shifts of finite type is investigated in \cite{BC-2015a}. In the mean time, the above problems remain open.

\bibliographystyle{amsplain}
\bibliography{../../grece}

\providecommand{\bysame}{\leavevmode\hbox to3em{\hrulefill}\thinspace}
\providecommand{\MR}{\relax\ifhmode\unskip\space\fi MR }
\providecommand{\MRhref}[2]{%
  \href{http://www.ams.org/mathscinet-getitem?mr=#1}{#2}
}
\providecommand{\href}[2]{#2}
\begin{thebibliography}{10}

\bibitem{AB-TCS2012}
N.~Aubrun and M.-P. B\'{e}al, \emph{Tree-shifts of finite type}, Theor. Comput.
  Sci. \textbf{459} (2012), 16--25.

\bibitem{AB-TCS2013}
\bysame, \emph{Sofic tree-shifts}, Theory Comput. Systems \textbf{53} (2013),
  621--644.

\bibitem{BC-2015a}
J.-C. Ban and C.-H. Chang, \emph{Tree-shifts: Part 2. entropy of tree-shifts of
  finite type}, submitted, 2015.

\bibitem{BHL+-2015}
J.-C. Ban, W.-G. Hu, S.-S. Lin, and Y.-H. Lin, \emph{Verification of mixing
  properties in two-dimensional shifts of finite type}, arXiv:1112.2471v2,
  2015.

\bibitem{BPS-TAMS2010}
M.~Boyle, R.~Pavlov, and M.~Schraudner, \emph{Multidimensional sofic shifts
  without separation and their factors}, Trans. Am. Math. Soc. \textbf{362}
  (2010), 4617--4653.

\bibitem{CJJ+-ETDS2003}
E.~Coven, A.~Johnson, N.~Jonoska, and K.~Madden, \emph{The symbolic dynamics of
  multidimensional tiling systems}, Ergodic Theory Dynam. Systems \textbf{23}
  (2003), 447--460.

\bibitem{Dev-1987}
R.~L. Devaney, \emph{An introduction to chaotic dynamical systems},
  Addison-Wesley, Redwood City, CA, 1987.

\bibitem{Hed-MST1969}
G.~A. Hedlund, \emph{Endomorphisms and automorphisms of full shift dynamical
  system}, Math. Systems Theory \textbf{3} (1969), 320--375.

\bibitem{JM-PAMS1999}
A.S.A. Johnson and K.M. Madden, \emph{The decomposition theorem for
  two-dimensional shifts of finite type}, Proc. Amer. Math. Soc. \textbf{127}
  (1999), 1533--1543.

\bibitem{Kit-1998}
B.~Kitchens, \emph{Symbolic dynamics. one-sided, two-sided and countable state
  {Markov} shifts}, Springer-Verlag, New York, 1998.

\bibitem{LY-AMM1975}
T.-Y. Li and J.~A. Yorke, \emph{Period three implies chaos}, Am. Math. Monthly
  \textbf{82} (1975), 985--992.

\bibitem{LM-1995}
D.~Lind and B.~Marcus, \emph{An introduction to symbolic dynamics and coding},
  Cambridge University Press, Cambridge, 1995.

\bibitem{LS-2002}
D.~Lind and K.~Schmidt, \emph{Symbolic and algebraic dynamical systems},
  Handbook of Dynamical Systems, vol.~1A, North-Holland, Amsterdam, 2002,
  p.~765–812.

\bibitem{MS-ETDS2009}
J.~M\"{u}ller and C.~Spandl, \emph{Embeddings of dynamical systems into
  cellular automata}, Ergodic Theory Dynam. Systems \textbf{29} (2009),
  165--177, Erratum. Ergodic Theory Dynam. Systems 30, 1271-1271.

\bibitem{VCC-PAMIIEEET2005}
J.L. Verdu-Mas, R.C. Carrasco, and J.~Calera-Rubio, \emph{Parsing with
  probabilistic strictly locally testable tree languages}, Pattern Analysis and
  Machine Intelligence, IEEE Transactions on \textbf{27} (2005), 1040--1050.

\end{thebibliography}

\end{document}